\def\versione{\label{vers.4}} 
\documentclass[11pt,a4paper]{article}

\usepackage{amsfonts}
\usepackage{color}
\usepackage{amsmath}
\usepackage{amsthm}
\usepackage{amssymb}
\usepackage{graphicx}
\usepackage{mathrsfs}

\newtheorem{theorem}{Theorem}[section]
\newtheorem{proposition}[theorem]{Proposition}
\newtheorem{lemma}[theorem]{Lemma}
\newtheorem{corollary}[theorem]{Corollary}

\newtheorem{examples}[theorem]{Examples}

\theoremstyle{definition}
\newtheorem{remark}[theorem]{Remark}

\numberwithin{equation}{section} \numberwithin{figure}{section}
\numberwithin{table}{section} \numberwithin{equation}{section}

\begin{document}

\title{\versione A generalization of  Kantorovich operators for convex compact subsets}
\author{Francesco Altomare, Mirella Cappelletti Montano, \\
Vita Leonessa and Ioan Ra\c{s}a \\
[0,8cm]}
\date{}
\maketitle

\begin{abstract}
\noindent
In this paper  we introduce and study a new sequence of positive linear operators acting on function spaces defined on a convex compact subset.
Their construction depends on a given Markov operator, a positive real number and a sequence of probability Borel measures. By considering special cases
of these parameters for particular convex compact subsets we obtain the classical Kantorovich operators defined in the one-dimensional and multidimensional setting together with several of their wide-ranging generalizations scattered in the literature. We investigate the approximation properties of these operators by also providing several estimates of the rate of convergence. Finally,
the preservation of Lipschitz-continuity as well as of convexity are discussed.
\end{abstract}



\section*{Introduction}

Over the last twenty years the interest and the relevance of the study of positive approximation processes on convex compact subsets have emerged with growing evidence, mainly
because of their useful connections with approximation problems both for functions defined on these domains and for the solutions of initial-boundary value differential problems.
In such a setting a prominent role is played by Bernstein-Schnabl operators which are generated by Markov operators.

In the monograph \cite{ACMLR2014} (see also \cite{ac}) a rather complete overview can be found about the main results in this field of researches together with their main applications.

In this paper we introduce and study a new sequence of positive linear operators acting on function spaces defined on a convex compact subset $K$ of some locally convex Hausdorff space. Their construction depends on a given Markov operator $T:\mathscr{C}(K) \rightarrow \mathscr{C}(K)$, a real number $a \geq 0$ and a sequence $(\mu_n)_{n \geq 1}$ of probability Borel measures on $K$.

By considering special cases of these parameters for particular convex compact subsets such as the unit interval or the multidimensional hypercube and simplex, we obtain all the Kantorovich operators defined on these settings together with several of other wide-ranging generalizations (\cite{ACML2010}, \cite{AL2006}, \cite{CV}, \cite{GRR}, \cite{Kantor}, \cite{N}, \cite{Stancu}, \cite{zhou}).

Moreover, for $a=0$, the new operators turn into the Bernstein-Schnabl operators and so, by means of the real continuous parameter $a \geq 0$, our sequence of operators represents, indeed, a link between the Bernstein operators ($a=0$) and the Kantorovich operators ($a=1$) on the classical one-dimensional and multidimensional domains where they are defined.

The paper is mainly devoted to investigate the approximation properties of the above mentioned operators in spaces of continuous functions and, for special settings, in $\mathscr{L}^p$-spaces
as well. Several estimates of the rate of convergence are also provided. In the final section we discuss some conditions under which these operators preserve Lipschitz-continuity or the convexity.

In a subsequent paper we intend to investigate whether and for which class of initial-boudary value differential problems, our operators, like Bern\-stein-Schnabl operators, can be useful to approximate the relevant solutions.

\section{Notation and preliminaries}

Throughout the paper we shall fix a locally convex Hausdorff space $X$ and a convex compact subset $K$ of $X$.

The symbol $X^{\prime}$ will denote   the dual space
of $X$ and the symbol $L(K)$ will stand for the space
\begin{equation}\label{lin} L(K):=\{\varphi_{|K}\,  \mid  \, \varphi \in X^{\prime}
\}.
\end{equation}

As usual, we shall denote by $\mathscr{C}(K)$ the space of all real-valued continuous functions on $K$;  $\mathscr{C}(K)$ is a Banach lattice if
endowed with the natural
(pointwise) ordering and the sup-norm $\|\cdot\|_{\infty}$.

Moreover, we shall denote by  $A(K)$  the space of all  continuous affine   functions
on $K$.

Whenever $X$ is the real Euclidean space $\mathbf{R}^d$  of dimension $d$ ($d\geq 1$), we shall denote by $\|\cdot\|_2$ the Euclidean norm on $\mathbf{R}^d$. Furthermore, we shall denote by $\lambda_d$ the Borel-Lebesgue measure on  $K\subset \mathbf{R}^d$  and by $|K|$ the measure of $K$ with respect to $\lambda_d$. Finally, for every $i=1,\dots,d$, $pr_i$ will
stand for the $i^{th}$ coordinate function on $K$, i.e.,
$pr_i(x):=x_i$ for every $x=(x_1,\dots,x_d)\in K$.


 Coming back to an arbitrary convex compact subset $K$, let $B_K$ be the $\sigma$-algebra of all Borel
subsets of $K$ and $M^+(K)$ (resp., $M^+_1(K)$) the cone of all
regular Borel measures on $K$ (resp., the cone of all regular Borel
probability measures on $K$).  For every $x \in K$ the symbol $\epsilon_x$ stands
 for the Dirac measure concentrated at $x$.

If $\mu \in M^+(K)$  and $1\leq p< +\infty$,
let us denote by $\mathscr{L}^p(K,\mu)$ the space of all
$\mu$-integrable in the $p^{th}$ power functions on $K$; moreover, we shall
denote by $\mathscr{L}^{\infty}(K,\mu)$ the space of all
$\mu$-essentially bounded measurable functions on $K$. In
particular, if $\mu=\lambda_d$,  then we shall use the symbols
$\mathscr{L}^p(K)$ and $\mathscr{L}^{\infty}(K)$.

From now on let $T: \mathscr{C}(K)\to \mathscr{C}(K)$ be a Markov operator, i.e., a positive linear operator on $\mathscr{C}(K)$ such that $T(\textbf{1})=\textbf{1}$, where the symbol $\textbf{1}$ stands for the function of constant value $1$ on $K$.
Moreover, let  $(\tilde{\mu}_x^T)_{x\in K}$ be the continuous selection of probability Borel measures on $K$
 corresponding to $T$ via the Riesz representation theorem, 
 i.e.,
\begin{equation}\label{borelmeas}
 \int_Kf\,d\tilde{\mu}_x^T=T(f)(x)\qquad (f\in  \mathscr{C}(K), x\in K).
\end{equation}

In \cite[Chapter 3]{ACMLR2014} (see also \cite[Chapter 6]{ac}) the authors introduced and studied the so-called Bernstein-Schnabl operators associated with the Markov operator $T$ and defined, for every $f\in \mathscr{C}(K)$ and $x\in K$,  as follows
\begin{equation}\label{bs}
B_n(f)(x)= \int_K\!\!\! \cdots \!\!\int_K f\left(\frac{x_1+ \ldots
+x_n}{n}\right) d\tilde{\mu}_x^T(x_1) \cdots d\tilde{\mu}_x^T(x_n).
\end{equation}
Note that, for every $n\geq 1$,  $B_n$ is a linear positive operator from $\mathscr{C}(K)$ into $\mathscr{C}(K)$, $B_n(\textbf{1})=\textbf{1}$ and hence $\|B_n\|=1$. Moreover, $B_1=T$.

The operators $B_n$ generalize the classical Bernstein operators on the unit interval, on multidimensional simplices and  hypercubes, and they share with them several preservation properties which are investigated in \cite{ACMLR2014} and \cite{ac}.

Further, if in addition we suppose that the Markov operator $T$ satisfies the following condition
\begin{equation}\label{ass1}
T(h)=h \qquad \textrm{ for every } h \in A(K),
\end{equation}
or, equivalently (see \eqref{borelmeas}),
\begin{equation}\label{ass1a}
\int_K h \, d\tilde{\mu}_x^T =h(x) \qquad \textrm{ for every } h \in
A(K) \textrm{ and } x \in K,
\end{equation}
then the sequence $(B_n)_{n\geq 1}$ is an approximation process on $\mathscr{C}(K)$. Namely,
for every $f \in \mathscr{C}(K)$,
\begin{equation}\label{approxim} \lim\limits_{n
\to \infty} B_n(f)=f
\end{equation}
uniformly on $K$.

Finally, for every $h, k \in A(K)$ and $n \geq
1$, the following useful formulae hold:
\begin{equation}\label{affini}
B_n(h)=h
\end{equation}
and
\begin{equation}\label{prop2}
B_n(hk)=\frac{1}{n}T(h k)+\frac{n-1}{n}hk.
\end{equation}
In particular,
\begin{equation}\label{affinisquares}
B_n(h^2)=\frac{1}{n}T(h^2)+ \frac{n-1}{n}h^2.
\end{equation}
(For a proof of \eqref{approxim}-\eqref{affinisquares} see  \cite[Theorem 3.2.1]{ACMLR2014} or \cite[Theorem 3.2]{ACMLRJFA2014}).

\section{Generalized Kantorovich operators}

In this section we introduce the main object of interest of the paper and we show some examples.

Let  $T: \mathscr{C}(K)\to \mathscr{C}(K)$ be a Markov operator  satisfying  condition  \eqref{ass1} (or, equivalently \eqref{ass1a}). 
Moreover, fix $a\geq 0$ and a sequence $(\mu_n)_{n\geq 1}$ of probability Borel measures on $K$.

Then, for every $n\geq 1$, we consider the positive linear operator $C_n$ defined by setting
\begin{equation}
\label{CnExprK}
    C_n(f)(x)\!=\!\!\int_K \!\!\!\cdots
    \!\!\int_K\! f\!\left(\!\frac{x_1+\ldots+x_n+ax_{n+1}}{n+a}\!\right)\!\!\,d\tilde{\mu}^T_x(x_1)\cdots
    d\tilde{\mu}^T_x(x_{n})d\mu_n(x_{n+1})
\end{equation}
for every $x\in K$ and for every $f\in\mathscr{C}(K)$.

The germ of the idea of the construction of \eqref{CnExprK} goes back to  \cite{AL2006} where some of the authors considered the particular case of the unit interval (see also \cite{AL}). Subsequently, in \cite{ACML2010} a natural generalization of $C_n$'s to the
multidimensional setting, i.e., to hypercubes and simplices, was presented obtaining, as a particular case, the multidimensional Kantorovich operators on these frameworks.

Here we develop such idea in full generality, obtaining a new class of positive linear operators which encompasses  not only several well-known approximation processes both in univariate and multivariate settings, but also new ones in finite and infinite dimensional frameworks as well.

Clearly, in the special case $a=0$, the operators $C_n$ correspond to the $B_n$ ones. Moreover, introducing
the auxiliary continuous function 
\begin{equation}\label{fn}
I_n(f)(x):=\int_K f\left(\frac{n}{n+a}\,x+\frac{a}{n+a}\,t\right) \,
d\mu_n(t) \qquad (f\in\mathscr{C}(K),\,x \in K),
\end{equation}
for every $n\geq 1$, then
\begin{equation}\label{newexp}
C_n(f)=B_n(I_n(f)).
\end{equation}

Therefore $C_n(f) \in \mathscr{C}(K)$ and the operator
$C_n:\mathscr{C}(K) \rightarrow \mathscr{C}(K)$, being linear
and positive, is continuous with norm equal to 1, because
$C_n(\mathbf{1})=\mathbf{1}$.

We point out that the operators $C_n$ are well-defined on the larger linear space of all Borel measurable functions $f:K\to\mathbf{R}$ for which the multiple integral in \eqref{CnExprK} is absolutely convergent.  This space contains, among other things, all the bounded Borel measurable functions on $K$ as well as a suitable subspace of $\bigcap\limits_{n \geq 1}\mathscr{L}^1(K,\mu_n)$.

Here, we prefer to not enter into more details to this respect by postponing a more detailed analysis in a subsequent paper. However, in Section 4 we shall discuss the approximation properties of these operators also in the setting of $\mathscr{L}^p(K)$-spaces in the particular cases where $K$ is a simplex or a hypercube of $\mathbf{R}^d$.

Note that assumption \eqref{ass1} is not essential in defining the operators $C_n$, but it will be needed in order to prove that $(C_n)_{n	\geq 1}$ is an approximation process on $\mathscr{C}(K)$ and on $\mathscr{L}^p(K)$   as we shall see in the next sections.  

By specifying the Markov operator $T$, i.e., the family of representing measures $(\tilde{\mu}_x^T)_{x\in K}$, the parameter $a\geq 0$  as well as the sequence of measures $(\mu_n)_{n\geq 1}$, we obtain several classes of approximating operators which can be tracked down in different papers.

In particular, when $a=1$, for a special class of $T$ and of the sequence $(\mu_n)_{n\geq 1}$ we get the Kantorovich operators on the unit interval, on simplices and on hypercubes (see the next examples).

Another interesting case is covered when $a$ is a positive integer. Indeed, given $\mu\in M_1^+(K)$, we may consider the measure $\mu_a\in M_1^+(K)$ defined by
\begin{equation*}
\int_K f\,d\mu_a:=\int_K\!\!\!\cdots\!\int_K f\left(\frac{y_1+\ldots+y_a}{a}\right)d\mu(y_1)\cdots d\mu(y_a)
\end{equation*}
($f\in\mathscr{C}(K)$) and hence, for $\mu_n:=\mu_a$ for every $n\geq 1$, the corresponding operators in \eqref{CnExprK} reduce to
\begin{equation}
\label{CnExprKbis}\begin{split}
    &C_n(f)(x)=\\ &\int_K \!\!\!\cdots
    \!\int_Kf\left(\frac{x_1+\ldots+x_n+y_1+\ldots+y_a}{n+a}\right)
    d\tilde{\mu}^T_x(x_1)\cdots
    d\tilde{\mu}^T_x(x_{n})d\mu(y_1)\cdots d\mu(y_a)
\end{split}
\end{equation}
($n \geq 1$, $f\in\mathscr{C}(K)$ $x \in K)$.

When $K=[0,1]$, for particular $T$ and $\mu$ we get the so-called Kantorovich operators of order $a$ (see \cite[Examples (A)]{CV} and \cite{N} for $a=2$). 

Further, we mention another particular simple case which, however, seems to be not devoid of interest. Assume $a>0$ and consider a sequence $(b_n)_{n\geq 1}$ in $X$ such that $b_n/a\in K$ for every $n\geq 1$. Then, setting $\mu_n:=\epsilon_{b_n/a}$ ($n\geq 1$), from \eqref{CnExprK} we get
\begin{equation}
\label{CnExprKter}
    C_n(f)(x)=\int_K \!\!\!\cdots
    \!\!\int_Kf\!\left(\!\frac{x_1+\ldots+x_n+b_n}{n+a}\!\right)\,d\tilde{\mu}^T_x(x_1)\cdots
    d\tilde{\mu}^T_x(x_{n})
    \end{equation}
($n\geq 1$, $f\in\mathscr{C}(K)$, $x\in K$).

We proceed to show more specific examples.

\begin{examples}\label{esCn}
\end{examples}
\noindent  $1.$   \indent Assume $K=[0,1]$ and consider the Markov operator  $T_1: \mathscr{C}([0,1])\to  \mathscr{C}([0,1])$ defined, for every $f\in  \mathscr{C}([0, 1])$ and $0\leq x\leq 1$, by
\begin{equation}\label{defT1}
T_1(f)(x):=(1-x)f(0)+xf(1).
\end{equation}

Then, the Bernstein-Schnabl operators associated with $T_1$ are the classical Bernstein operators
\begin{equation}
\label{new2.7}
B_n(f)(x):=\sum_{k=0}^n \binom{n}{k}x^k(1-x)^{n-k}f\left(\frac{k}{n}\right)
\end{equation}
($n\geq 1$, $f\in\mathscr{C}([0,1])$, $x\in [0,1]$), and, considering $a\geq 0$ and $(\mu_n)_{n\geq 1}$ in $M_1^+([0,1])$, from \eqref{CnExprK} and \eqref{newexp} we get
\begin{equation}
\label{new2.8}
C_n(f)(x)=\sum_{k=0}^n\binom{n}{k} x^k(1-x)^{n-k}\int_0^1f\left(\frac{k+as}{n+a}\right)d\mu_n(s)
\end{equation}
($n\geq 1$, $f\in\mathscr{C}([0,1])$, $x\in [0,1]$). In particular, if all the $\mu_n$ are equal to the Borel-Lebesgue measure $\lambda_1$ on $[0,1]$ and $a=1$, then formula \eqref{new2.8} gives the classical Kantorovich operators  (\cite{Kantor}; see also \cite[Subsection 5.3.7]{ac}). Moreover, as already remarked, for $a=0$ we obtain the Bernstein operators; thus, by means of \eqref{new2.8}, we obtain a link between these fundamental sequences of approximating operators in terms of a continuous parameter $a\in [0,1]$. Special cases of operators \eqref{new2.8} have been also considered in \cite{AL2006} and \cite{GRR} and we omit the details for the sake of brevity.

When $a$ is a positive integer, from \eqref{CnExprKbis} we obtain 
\begin{equation}\begin{split}
\label{new2.9}
&C_n(f)(x)\\
&=\sum_{k=0}^n\binom{n}{k} x^k(1-x)^{n-k}\int_0^1\!\!\!\cdots\!\!\int_0^1 f\left(\frac{k+y_1+\cdots+y_a}{n+a}\right)d\mu(y_1)\cdots d\mu(y_a)\end{split}
\end{equation}
($n\geq 1$, $f\in\mathscr{C}([0,1])$, $0\leq x\leq 1$), $\mu\in M_1^+([0,1])$ being fixed. When $\mu$ is the Borel-Lebesgue measure on $[0,1]$ we obtain the previously mentioned Kantorovich operators of order $a$ (\cite[Examples (A)]{CV} and \cite{N} for $a=2$). 
Finally, from \eqref{CnExprKter} and with $b_n\leq a$ ($n\geq 1$), we get the operators
\begin{equation}
\label{new2.10}
C_n(f)(x)=\sum_{k=0}^n\binom{n}{k} x^k(1-x)^{n-k} f\left(\frac{k+b_n}{n+a}\right)
\end{equation}
($n\geq 1$, $f\in\mathscr{C}([0,1])$, $x\in [0,1]$), which have been first considered in \cite{Stancu} for a constant sequence $(b_n)_{n\geq 1}$. \\ \\
$2.$  \indent Let $Q_d:=[0,1]^d$, $d\geq 1$,  and consider the Markov operator $S_d:\mathscr{C}(Q_d) \rightarrow \mathscr{C}(Q_d)$ defined by
\begin{equation}\label{projh1}
S_d(f)(x):=\sum_{h_1, \ldots, h_d=0}^1 f( \delta_{h_1 1}, \ldots,
\delta_{h_d 1}) x_1^{h_1}(1-x_1)^{1-h_1} \cdots
x_d^{h_d}(1-x_d)^{1-h_d}
\end{equation}
($f \in \mathscr{C}(Q_d)$, $x=(x_1, \ldots, x_d) \in Q_d)$, where
$\delta_{i j}$ stands for the Kronecker symbol.

In this case, the Bernstein-Schnabl operators associated with $S_d$ are the classical Bernstein operators on $Q_d$ defined by
 \begin{equation}\label{bernhy}
  \begin{split}
    &B_n(f)(x)=\sum\limits_{h_1,\dots,h_d=0}^n
  \,\prod_{i=1}^d {\binom{n}{h_i}} x_i^{h_i}(1-x_i)^{n-h_i} f\left(\dfrac{h_1}{n}, \ldots, \dfrac{h_d}{n}\right)
   \end{split}
 \end{equation}
($n\geq 1$, $f\in \mathscr{C}(Q_d)$, $x=(x_1,\ldots,x_d)\in Q_d)$.

Then, taking \eqref{newexp}  into account, the operators $C_n$ given by \eqref{CnExprK} become
\begin{equation}\label{new2.12}
\begin{split}
&C_n(f)(x)=\sum\limits_{h_1,\dots,h_d=0}^n
  \,\prod_{i=1}^d {\binom{n}{h_i}} x_i^{h_i}(1-x_i)^{n-h_i} \\& \times \int_{Q_d} f\left(\dfrac{h_1+as_1}{n+a}, \ldots, \dfrac{h_d+as_d}{n+a}\right) \, d\mu_n(s_1, \ldots, s_d)
  \end{split}
  \end{equation}
($n\geq 1$, $f\in \mathscr{C}(Q_d)$, $x=(x_1,\ldots,x_d)\in Q_d)$.

 When all the $\mu_n$ coincide with the Borel-Lebesgue measure $\lambda_d$ on $Q_d$ and $a=1$, the operators $C_n$ turn into a generalization of Kantorovich operators introduced in 
\cite{zhou}. Another special case of \eqref{new2.12} has been studied in \cite{ACML2010}.\\ \\
$3.$  \indent 
 Denote by $K_d$ the canonical simplex in
$\mathbf{R}^d$, $d\geq 1$, i.e.,
\begin{equation}\label{simplexd}
K_d:=\left\{(x_1, \ldots , x_d) \in \mathbf{R}^d \, | \, x_i \geq
 0 \,(i=1, \ldots, d) \, \textrm{and} \, \sum\limits_{i=1}^d x_i \leq 1\right\},
\end{equation}
 and consider the canonical Markov operator  $T_d :\mathscr{C}(K_d) \rightarrow \mathscr{C}(K_d)$  defined by
\begin{equation}\label{pros}
T_d(f)(x):=\left(1-\sum\limits_{i=1}^d
x_i\right)f(0)+\sum\limits_{i=1}^d x_i f(e_i)
\end{equation}
($f \in  \mathscr{C}(K_d)$, $ x=(x_1, \ldots, x_d) \in K_d$)
where, for every $i=1, \ldots,d,$ $e_i:=(\delta_{i j})_{1 \leq j
\leq d}$,  $\delta_{i j}$ being the Kronecker symbol (see, for instance, \cite[Section 6.3.3]{ac}).

The Bernstein-Schnabl operators associated with $T_d$ are the classical Bernstein operators on $K_d$ defined by
\begin{equation}
\label{bernsimplex}
\begin{split}
 &B_n(f)(x)=
 \sum\limits_{h_1, \ldots, h_d =0, \ldots, n \atop
h_1+\ldots+h_d \leq n} \frac{n!}{h_1!\cdots h_d!(n-h_1- \ldots-h_d)!} x_1^{h_1}\cdots x_d^{h_d} \\& \times \left(1- \sum\limits_{i=1}^d
x_i\right)^{n-\sum\limits_{i=1}^d h_i} f\left(\dfrac{h_1}{n}, \ldots, \dfrac{h_d}{n}\right)
\end{split}
\end{equation}
($n\geq 1$, $f\in \mathscr{C}(K_d)$, $x=(x_1,\ldots,x_d)\in K_d$).

By using again \eqref{newexp}, we obtain
\begin{equation}
\label{new2.15}
\begin{split}
 &C_n(f)(x)=\\
 &\sum\limits_{h_1, \ldots, h_d =0, \ldots, n \atop
h_1+\ldots+h_d \leq n} \frac{n!}{h_1!\cdots h_d!(n-h_1- \ldots-h_d)!} x_1^{h_1}\cdots x_d^{h_d}\left(1- \sum\limits_{i=1}^d
x_i\right)^{n-\sum\limits_{i=1}^d h_i}
\\&   \times    \int_{K_d} f\left(\frac{h_1+as_1}{n+a}, \frac{h_2+as_2}{n+a},\ldots,\frac{h_d+as_d}{n+a}\right)\,d\mu_n(
    s_1,\ldots,s_d)
\end{split}
\end{equation}
($n\geq 1$, $f\in \mathscr{C}(K_d)$, $x=(x_1,\ldots,x_d)\in K_d$).

When all the $\mu_n$ are equal to the Borel-Lebesgue measure $\lambda_d$ on $K_d$ and $a=1$, these operators are referred to as the Kantorovich operators on $ \mathscr{C}(K_d)$ and were introduced in \cite{zhou}. Another particular case of \eqref{new2.15} has been investigated in  \cite[Section 3]{ACML2010}. 

For the sake of brevity we omit to describe in the setting of $ \mathscr{C}(Q_d)$ and $ \mathscr{C}(K_d)$ the operators corresponding to \eqref{CnExprKbis} and \eqref{CnExprKter}.

\section{Approximation properties in $\mathscr{C}(K)$}

 In this section we present some approximation properties
of the sequence $(C_n)_{n\geq 1}$ on $\mathscr{C}(K)$, showing
several estimates of the rate of convergence.

In order to prove that the sequence $(C_n)_{n\geq 1}$ is a
(positive) approximation process on $\mathscr{C}(K)$ we need the
following preliminary result.

\begin{lemma}\label{lemma1}
Let $(C_n)_{n\geq 1}$ be the sequence of operators defined by
(\ref{CnExprK}) and associated with a Markov operator satisfying \eqref{ass1} (or, equivalently \eqref{ass1a}). Then, for every $h,k\in A(K)$,
\begin{equation}\label{Cn(h)}
    C_n(h)=\dfrac{a}{n+a}\int_K h\,d
    \mu_n\cdot \mathbf{1}+\dfrac{n}{n+a}h
\end{equation}
and
\begin{equation}\label{Cn(hk)}\begin{split}
   &C_n(hk)=\dfrac{a^2}{(n+a)^2}\int_K hk\,d\mu_n\cdot \mathbf{1}+\dfrac{na}{(n+a)^2}\int_K h\,d\mu_n\cdot
    k\\&+\dfrac{na}{(n+a)^2}\int_K k\,d\mu_n\cdot
    h+\dfrac{n^2}{(n+a)^2}B_n(hk).
  \end{split}  \end{equation}
In particular,
\begin{equation}\label{Cn(h^2)}
    C_n(h^2)=\dfrac{a^2}{(n+a)^2}\int_K h^2\,d\mu_n\cdot \mathbf{1}+\dfrac{2na}{(n+a)^2}\int_K h\,d\mu_n\cdot
    h+\dfrac{n^2}{(n+a)^2}B_n(h^2).
\end{equation}

\end{lemma}

\begin{proof}
By \eqref{fn}, if $h \in A(K)$, then
$\displaystyle{I_n(h)=\frac{n}{n+a}h+\frac{a}{n+a} \int_K h \
d\mu_n}\cdot\mathbf{1}$ and hence \eqref{Cn(h)} follows taking \eqref{newexp} 
and \eqref{affini}  into account.

Analogously, \eqref{Cn(hk)}  is a consequence of the identity
\begin{equation*}\begin{split}
&I_n(hk)=\frac{n^2}{(n+a)^2}hk + \frac{a^2}{(n+a)^2} \int_K hk \
d\mu_n\cdot\mathbf{1}\\
& + \frac{na}{(n+a)^2} \left(\int_K h \ d\mu_n\right) k+\frac{na}{(n+a)^2} \left(\int_K k \ d\mu_n\right) h
\end{split}
\end{equation*}
and \eqref{prop2}.
Formula  \eqref{Cn(h^2)} is a direct consequence of    \eqref{Cn(hk)}.
\end{proof}

The following approximation result holds.

\begin{theorem}\label{ApproxC(K)}
Under assumption \eqref{ass1}, for every $f\in \mathscr{C}(K)$
\begin{equation}\label{approxCnC(K)}
    \lim_{n\rightarrow \infty} C_n(f)=f\quad \textit{uniformly on }\,\,K.
\end{equation}
\end{theorem}
\begin{proof}
First observe that the space $A(K)$ contains the constant functions and separates the points of $K$ by the Hahn-Banach theorem. Then,
according to \cite[Example 3 of Theorem 4.4.6]{ac} (see also \cite[Theorem 1.2.8]{ACMLR2014}),  $A(K)\cup A(K)^2$ is a Korovkin subset for $\mathscr{C}(K)$. Hence the
claim will be proved if we show that, for every $h\in A(K)$, 
\begin{equation*}
\lim\limits_{n \rightarrow \infty} C_n(h)=h \ \quad \ \textrm{and}
\ \quad \ \lim\limits_{n \rightarrow \infty} C_n(h^2)=h^2
\end{equation*}
uniformly on $K$. All these assertions
follow from   Lemma \ref{lemma1}, observing that the sequences
$\left( \int_K h \ d\mu_n\right)_{n \geq 1}$ and $\left( \int_K h^2
\ d\mu_n\right)_{n \geq 1}$ are bounded for every $h \in A(K)$.
\end{proof}

Now we present some quantitative estimates of the rate of convergence in \eqref{approxCnC(K)} by means of suitable moduli of continuity both in the finite and infinite dimensional settings.

To this end we need to recall some useful definitions.

We begin with the finite dimensional case, i.e.,
  $K$ is a convex compact subset of $\mathbf{R}^d$, $d\geq 1$.  Then
we can estimate the rate of uniform convergence of the
sequence $(C_n(f))_{n \geq 1}$ to $f$   by means of the first and the second
moduli of continuity,
 respectively defined as
\begin{equation}\label{om1}
    \omega (f, \delta):=\sup \{|f(x)-f(y)| \; | \; x,y \in K , \|x-y\|_2 \leq \delta \}
\end{equation}
and
\begin{equation}\label{om2}
    \omega_2(f, \delta):=\sup \left\{\left|f(x)-2f\left(\frac{x+y}{2}\right)+f(y)\right| \; | \; x,y \in K , \|x-y\|_2 \leq 2\delta \right\},
\end{equation}
for any $f \in \mathscr{C}(K)$ and $\delta >0$.

In the general case of a locally convex Hausdorf space $X$ (not necessarily of finite dimension) we shall use the total modulus of continuity which we are going to define.

First, if $m \geq 1$, $h_1, \ldots, h_m \in L(K)$ (see \eqref{lin}) and
$\delta >0$, we set
\begin{equation}\label{khm}
H(h_1, \ldots, h_m, \delta):=\left\{(x,y) \in K \times K \, \mid \,
\sum_{j=1}^m (h_j(x)-h_j(y))^2 \leq \delta^2\right\}.
\end{equation}

Fix a bounded function $f :K\to \mathbf{R}$; the  modulus of continuity of $f$
with respect to $h_1, \ldots, h_m$ is defined as
\begin{equation}\label{omegahm}
\omega(f; h_1, \ldots, h_m, \delta):= \sup\{|f(x)-f(y)| \ \mid \
(x,y) \in H(h_1, \ldots, h_m, \delta)\}.
\end{equation}

Moreover, we define the  total modulus of continuity of  $f$ as
\begin{equation}\label{omegaf}
\begin{split}
&\Omega(f, \delta):=\\
&=\inf \Big\{\omega(f; h_1, \ldots, h_m, \delta) \
\mid \ m \geq 1, \ h_1, \ldots, h_m \in L(K),
\left|\left|\sum_{j=1}^m h_j^2\right|\right|_\infty =1\Big\}\\&= \inf\Big\{\omega(f; h_1, \ldots, h_m, 1) \
\mid \ m \geq 1, \ h_1, \ldots, h_m \in L(K),
\left|\left|\sum_{j=1}^m h_j^2\right|\right|_\infty =\frac{1}{\delta^2}\Big\}.
\end{split}
\end{equation}

If $X=\mathbf{R}^d$  there is a simple relationship
between $\Omega(f, \delta)$ and the (first) modulus of continuity
$\omega(f, \delta)$ defined by \eqref{om1}; indeed,
\begin{equation}\label{omegapr}
\omega(f; pr_1, \ldots, pr_d, \delta)=\omega(f, \delta),
\end{equation}
so that, setting $r(K):=\max\{\|x\|_2 \ \mid \ x \in K\}$,
\begin{equation}\label{relomega}
\Omega(f, \delta) \leq \omega(f, \delta r(K)),
\end{equation}
the last inequality being an equality if $d=1$.

By using Proposition 1.6.5 in \cite{ACMLR2014} (see also \cite[Proposition 5.1.4]{ac})  we get the following result.

\begin{proposition}\label{prop3.3}
For every $n\geq 1$ and $f\in  \mathscr{C}(K)$,
\begin{equation}
\label{stimaOmegaCn}
\|C_n(f)-f\|_{\infty}\leq 2\Omega\left(f,\sqrt{\frac{4a^2+1}{n+a}}\right).
\end{equation}
\end{proposition}

\begin{proof}
Since $C_n(\mathbf{1})=\mathbf{1}$ ($n\geq 1$), we can apply estimate (1.6.14) in   \cite[Proposition 1.6.5]{ACMLR2014} obtaining, for every $n\geq 1$, $f\in \mathscr{C}(K)$, $x\in K$ and $\delta>0$, and for every $h_1,\ldots,h_m\in L(K)$, $m\geq 1$,
\[\tag{1}\label{formula1}
|C_n(f)(x)-f(x)|\leq \left(1+\frac{1}{\delta^2} \sum_{j=1}^m\mu(x,C_n,h_j)\right)\omega(f;h_1,\ldots,h_m,\delta),
\]
where $\mu(x,C_n,h_j)=C_n((h_j-h_j(x)\mathbf{1})^2)(x)$, $j=1,\ldots,m$. Therefore,
\[\tag{2}\label{formula2}
|C_n(f)(x)-f(x)|\leq \left(1+\tau_n(\delta,x)\right)\Omega(f,\delta),
 \]
where $\Omega(f,\delta)$ is the total modulus of continuity (see \eqref{omegaf}) and
\[
\begin{split}
\tau_n(\delta, x):= \sup\left\{\sum_{j=1}^m \mu(x,C_n,h_j)(x) \ \mid \
m \geq 1,  h_1, \ldots, h_m \in L(K) \right. \\ \left. \textrm{ and
} \left|\left|\sum_{j=1}^m h_j^2\right|\right|_\infty
=\frac{1}{\delta^2}\right\}.
\end{split}
\]
Fix $h\in L(K)$. Keeping \eqref{Cn(h)}, \eqref{Cn(h^2)} and \eqref{affinisquares}  in mind, we have
\[
\begin{split}
&C_n((h-h(x)\mathbf{1})^2)(x)=C_n(h^2)(x)-2h(x)C_n(h)(x)+h^2(x)\\
&=\frac{a^2}{(n+a)^2}\int_Kh^2\,d\mu_n+\frac{2na}{(n+a)^2}h(x) \int_Kh\,d\mu_n +\frac{n^2}{(n+a)^2} B_n(h^2)(x)\\
&-\frac{2a}{n+a}h(x)\int_Kh\,d\mu_n-\frac{2n}{n+a}h^2(x)+h^2(x)\\
&= \frac{a^2}{(n+a)^2}\int_Kh^2\,d\mu_n- \frac{2a^2}{(n+a)^2} h(x) \int_Kh\,d\mu_n+\frac{n}{(n+a)^2} T(h^2)(x)\\
&+ \frac{a^2-n}{(n+a)^2} h^2(x)\\
&\leq  \frac{a^2}{(n+a)^2}\int_Kh^2\,d\mu_n+\frac{2a^2}{(n+a)^2}\left|h(x) \int_Kh\,d\mu_n\right|+\frac{n}{(n+a)^2} T(h^2)(x)\\
&+ \frac{a^2}{(n+a)^2} h^2(x).
\end{split}
\]
Then, using Cauchy-Schwarz and Jensen inequalities, we get
\[
\begin{split}
&\sum_{j=1}^mC_n((h_j-h_j(x)\mathbf{1})^2)(x)\leq \\
&\leq  \frac{a^2}{(n+a)^2}\int_K\sum_{j=1}^mh_j^2\,d\mu_n+\frac{2a^2}{(n+a)^2}\sum_{j=1}^m\left|h_j(x) \int_Kh_j\,d\mu_n\right|\\
&+\frac{n}{(n+a)^2} T\left(\sum_{j=1}^mh_j^2\right)(x)
+ \frac{a^2}{(n+a)^2}\sum_{j=1}^m h_j^2(x)\\
&\leq  \frac{a^2}{(n+a)^2}\left\|\sum_{j=1}^mh_j^2\right\|_{\infty}+\frac{2a^2}{(n+a)^2}\left(\sum_{j=1}^m h_j^2(x) \right)^{1/2}\left(\sum_{j=1}^m\left|\int_Kh_j\,d\mu_n\right|^2\right)^{1/2}\\
&+\frac{n}{(n+a)^2}\left\|\sum_{j=1}^mh_j^2\right\|_{\infty}
+ \frac{a^2}{(n+a)^2}\left\|\sum_{j=1}^m h_j^2\right\|_{\infty}.
 \end{split}
\]
Then, for every $\delta>0$ and for every $h_1, \ldots h_m \in L(K)$, $m \geq 1$, such that $\|\sum_{j=1}^m h_j^2\|_\infty=1/\delta^2$,
\[
\begin{split}
&\sum_{j=1}^mC_n((h_j-h_j(x)\mathbf{1})^2)(x)\leq\\
  &\leq \frac{2a^2}{(n+a)^2}\frac{1}{\delta^2}+\frac{2a^2}{(n+a)^2}\left(\left\|\sum_{j=1}^m h_j^2\right\|_{\infty}\right)^{1/2}\left(\sum_{j=1}^m\int_K h_j^2\,d\mu_n\right)^{1/2}\\
&+\frac{n}{(n+a)^2}\frac{1}{\delta^2}
\leq  \frac{2a^2}{(n+a)^2}\frac{1}{\delta^2}+\frac{2a^2}{(n+a)^2}\left\|\sum_{j=1}^m h_j^2\right\|_{\infty}+\frac{n}{(n+a)^2}\frac{1}{\delta^2}\\
&= \frac{4a^2+n}{\delta^2(n+a)^2} \leq  \frac{4a^2+1}{(n+a)\delta^2}.
 \end{split}
\]

From \eqref{formula2} we infer that
\begin{equation*}
|C_n(f)(x)-f(x)| \leq \left(1+ \dfrac{4a^2+1}{(n+a) \delta^2}\right) \Omega(f, \delta)
\end{equation*}
for every $\delta>0$.

Setting $\delta=\sqrt{{(4a^2+1)}/{(n+a)}}$ we get the claim.

\end{proof}

We proceed to establish some estimates in the finite dimensional case. 
First we state the following result involving $\omega(f,\delta)$.

In the sequel $e_2:K \longrightarrow \mathbf{R}$ will denote the function 
\begin{equation*}\label{e2}
e_2(x):=||x||_2^2= \sum_{i=1}^dpr_i^2(x) \qquad (x \in K)
\end{equation*}
and, for a given $x \in K$, 
$d_x:K \longrightarrow
\mathbf{R}$ will stand for the function
\begin{equation}\label{formula1.4.3}
d_x(y):=||y-x||_2\ \qquad (y \in K).
\end{equation}

\begin{proposition}
For every  $f \in \mathscr{C}(K)$, $n \geq 1$ and $x \in K$,
\begin{equation}\label{omegapunt}
|C_n(f)(x)-f(x)| \leq 2 \omega\left(f, \frac{1}{n+a} \sqrt{a^2\int_K
d_x^2 \ d\mu_n + n (T(e_2)(x)-e_2(x))}\right).
\end{equation}

Moreover,
 \begin{equation}\label{omegaunif}
     \| C_n(f)-f\|_\infty  \leq 2\omega\left(f, \frac{\max(a\delta(K)^2,
     ||T(e_2)-e_2||_\infty)}{\sqrt{n+a}}\right),
     \end{equation}
where $\delta(K):=\sup \{||x-y||_2 \ | \ x,y \in K\}$.
 \end{proposition}
\begin{proof}
Considering the coordinate functions $pr_i$, $i=1, \ldots, d$,
clearly, for every $x \in K$,
$$d_x^2= \sum_{i=1}^d (pr_i-pr_i(x))^2.$$

Therefore,  from  formula \eqref{formula1} of the proof of Proposition \ref{prop3.3} and from  \eqref{omegapr} (see also \cite[Proposition 5.1.4 and
(5.1.13)]{ac}), it follows that, for any $n \geq 1$ and
$\delta >0$,
$$|C_n(f)(x)-f(x)| \leq \left(1+
\frac{1}{\delta^2}C_n(d_x^2)(x)\right) \omega(f, \delta).$$

If $C_n(d_x^2)(x)=0$ then, letting $\delta \rightarrow 0^+$, we get
$C_n(f)(x)=f(x)$ and, in this case, \eqref{omegapunt} is obviously
satisfied.

If $C_n(d_x^2)(x) >0$, then for $\delta:=\sqrt{C_n(d_x^2)(x)}$, we
obtain
$$|C_n(f)(x)-f(x)| \leq 2 \omega(f, \sqrt{C_n(d_x^2)(x)}).$$

On the other hand, by applying Lemma \ref{lemma1} to each $pr_i$,
$i=1, \ldots, d$, we have that
$$C_n(d_x^2)(x)=\frac{a^2}{(n+a)^2} \int_K d_x^2 \ d\mu_n +
\frac{n}{(n+a)^2} (T(e_2)(x)-e_2(x))$$ and hence \eqref{omegapunt}
holds. Clearly \eqref{omegaunif} follows from \eqref{omegapunt}.
\end{proof}


In the sequel we shall show a further estimate of the rate of
convergence in (\ref{approxCnC(K)}) by means of
$\omega_2(f,\delta)$.

To this end, according to \cite{berens}, we denote by
\begin{equation}\label{lambdainfty}
    \lambda_{n,\infty}:=\max_{0\leq i\leq
    d+1}\|C_n(\varphi_i)-\varphi_i\|_{\infty},
\end{equation}
where the functions $\varphi_i$ are defined by
\begin{equation}\label{phyfunctions}
\varphi_0:=\mathbf{1},\,\varphi_i:=pr_i\,\,(i=1,\ldots,d)
\,\,\,\,\textrm{and}\,\,\,\,\varphi_{d+1}:=\sum\limits_{i=1}^dpr_i^2.
\end{equation}

Then we have the following result.

\begin{proposition}
For every $f\in \mathscr{C}(K)$ and $n \geq 1$,
\begin{equation*}\label{berens1}
    \|C_n(f)-f\|_{\infty}\leq
    C\left(\frac{M}{n+a}\|f\|_{\infty}+\omega_2\left(f,\sqrt{\frac{M}{n+a}}\right)\right),
\end{equation*}
where the constants $C$ and $M$ do not depend on $f$. 
\end{proposition}
\begin{proof}
Since every convex bounded set has the cone property (see \cite[p.
66]{adams}), from   \cite[Theorem 2']{berens} we infer that, for every $f \in \mathscr{C}(K)$ and $n \geq 1$,
\begin{equation*}
\|C_n(f)-f\|_{\infty}\leq
    C(\lambda_{n,\infty}\|f\|_{\infty}+\omega_2(f,\lambda_{n,\infty}^{1/2})),
    \end{equation*}
where $\lambda_{n,\infty}$ is defined by  \eqref{lambdainfty} and the constant $C$ does not depend on $f$. In order  to estimate $\lambda_{n,\infty}$, note that, for every $n \geq 1$, $C_n(\mathbf{1})=\mathbf{1}$ and, for
$i=1, \ldots, d$ and $x=(x_1, \ldots, x_d) \in K$, taking
\eqref{Cn(h)} into account,
\begin{equation*}
|C_n(pr_i)(x)-pr_i(x)| =\frac{a}{n+a} \left|\int_K pr_i \, d\mu_n
-x_i\right| \leq \frac{2a\|x\|_{2}}{n+a};
\end{equation*}
hence
\[
\|C_n(\varphi_i)-\varphi_i\|_{\infty}\leq \frac{2ar(K)}{n+a}
\]
where $r(K):=\max\{\|x\|_2 \ \mid \ x \in K\} $.

On the other hand, by virtue of \eqref{Cn(h^2)} and \eqref{affinisquares}, for every $n \geq 1$, $i=1, \ldots d$ and $x=(x_1, \ldots, x_d)
\in K$ we get
\begin{equation*}
\begin{split}
& C_n(pr^2_i)(x)-pr^2_i(x) = \frac{1}{(n+a)^2} \left(a^2\int_K
pr_i^2 \, d\mu_n \right. \\ &\left.+ 2na \left(\int_K pr_i \, d\mu_n\right) x_i + n
(T(pr_i^2)(x)-x_i^2) -a (2n+a) x_i^2\right).
\end{split}
\end{equation*}

Therefore, 

\begin{equation*}\begin{split}
&|C_n(\varphi_{d+1})(x)- \varphi_{d+1}(x)|\leq \sum_{i=1}^d\left|C_n(pr^2_i)(x)-pr^2_i(x)\right| \leq   \frac{1}{(n+a)^2} \\& \times \left\{a^2 r(K)^2+ 2a d n \, r(K)^2
+n\sum_{i=1}^d |T(pr_i^2)(x)-x_i^2| + a(2n+a) r(K)^2\right\}
\\
&\leq
\frac{1}{(n+a)^2}\left\{(2a(n+a)+2adn +n) r(K)^2 +n\sum_{i=1}^d T(pr_i^2)(x)\right\}\\
&=
\frac{1}{(n+a)^2}\left\{(2a(n+a)+2adn +n)r(K)^2+nT\left(\sum_{i=1}^d pr_i^2\right)(x)\right\}
\\&
\leq
\frac{r(K)^2}{(n+a)^2}\left\{2a(n+a)+2adn+2n)\right\}.
\end{split}
\end{equation*}

Consequently,
\begin{equation*}
\begin{split}
\|C_n(\varphi_{d+1}) - \varphi_{d+1}\|_\infty  \leq
\frac{2a+2ad+2}{n+a} r(K)^2
\end{split}\end{equation*}
and, if we set
$M:=\max\left\{ 2a r(K), (2a+2ad+2) r(K)^2\right\}$, we get
\begin{equation*}
\lambda_{n,\infty} \leq \frac{M}{n+a};
\end{equation*}
this completes the proof.
\end{proof}

\section{Approximation properties in $\mathscr{L}^p$-spaces}
 
\noindent In this section we shall investigate particular subclasses of the operators $C_n$, $n \geq 1$, which are well-defined in $\mathscr{L}^p(K)$-spaces, $1 \leq p<+\infty$.
This analysis will be carried out in the special cases where $K$ is the $d$-dimensional unit hypercube (in particular, the unit interval) and the $d$-dimensional simplex.

In order to estimate the rate of convergence we recall here the definition of some moduli of smoothness. Let $K$ be a convex compact subset of $\mathbf{R}^d$, $d \geq 1$, having non-empty interior.

If $f:K \rightarrow \mathbf{R}$ is a Borel measurable bounded function and $\delta>0$, we define the (multivariate) averaged modulus of smoothness of the first order for $f$ and step $\delta$ in 
$\mathscr{L}^p$-norm, $1 \leq p <+\infty$, as
\begin{equation}\label{tau}
\tau(f, \delta)_p:=\|\omega(f, \cdot; \delta)\|_p
\end{equation}
where, for every $x \in K$,
\begin{equation}\label{omegalp}
\begin{split}
&\omega(f, x; \delta):=\\&=\sup\{|f(t+h)-f(t)| \mid t, t+h \in K, \|t-x\|_2 \leq \delta/2, \|t+h-x\|_2 \leq \delta/2\}.
\end{split}
\end{equation}

Furthermore, if $f \in \mathscr{L}^p(K)$, $1 \leq p<+\infty$ and $\delta>0$, the (multivariate) modulus of smoothness for $f$ of order $k$ and step $\delta$ in 
$\mathscr{L}^p$-norm is defined by
\begin{equation}\label{omegapk}
\omega_{k, p}(f, \delta):=\sup_{0 < |h| \leq \delta} \left(\int_K |\Delta_h^k f(x)|^p \, dx\right)^{1/p}
\end{equation}
where, for $x \in K$,
\begin{equation}\label{omegahkf}
\Delta_h^kf(x):=\left\{\begin{array}{ll}
\sum\limits_{l=0}^k (-1)^{k-l} \binom{k}{l} f(x+lh) & \textrm{ if } x+hk \in K,\\ \\
0 & \textrm{ otherwise}.
\end{array}
\right.
\end{equation}

 We now restrict ourselves to consider the $d$-dimensional unit hypercube $Q_d:=[0,1]^d$, $d\geq 1$. For every $n\geq 1$, $h=(h_1,\ldots,h_d)\in\{0,\ldots,n\}^d$ and $x=(x_1,\ldots,x_d)\in Q_d$, we set
 \begin{equation}\label{new3.1}
 P_{n,h}(x):=\prod_{i=1}^d \binom{n}{h_i} x_i^{h_i}(1-x_i)^{n-h_i};
\end{equation}
then
\begin{equation}\label{new3.3}
 P_{n,h}\geq 0\quad\textrm{and}\quad \sum_{h\in\{0,\ldots,n\}^d }  P_{n,h}=1\quad\textrm{on}\,\,Q_d.
\end{equation}
Further,
 \begin{equation}\label{new3.4}
\int_{Q_d}P_{n,h}(x)\,dx=\prod_{i=1}^d\binom{n}{h_i}\int_0^1 x_i^{h_i}(1-x_i)^{n-h_i}\,dx_i=\frac{1}{(n+1)^d}.
\end{equation}

Moreover, for any given $a\geq 0$, we set
 \begin{equation}\label{new3.2}
 Q_{n,h}(a):=\prod_{i=1}^d  \left[\frac{h_i}{n+a},\frac{h_i+a}{n+a} \right]\subset Q_d;
\end{equation}
in particular,
\begin{equation}\label{new3.5}
\bigcup_{h\in\{0,\ldots,n\}^d }  Q_{n,h}(a)=Q_d.
\end{equation}

Consider the operators $C_n$, $n\geq 1$, defined by \eqref{new2.12} with all the measures $\mu_n$ equal to Borel-Lebesgue measure on $Q_d$. In this case the operators
$C_n$ are well defined on $\mathscr{L}^1(Q_d)$ as well and the operator $S_d$ satisfies \eqref{ass1}. Moreover,  if $f\in\mathscr{L}^1(Q_d)$ and $x\in Q_d$,
 \begin{equation}\label{new3.6}
 C_n(f)(x)=\sum_{h\in\{0,\ldots,n\}^d }P_{n,h}(x)\int_{Q_d} f\left(\frac{h+au}{n+a}\right)du;
 \end{equation}
in particular,  if $a>0$, then
  \begin{equation}\label{new3.61}
 C_n(f)(x)=\sum_{h\in\{0,\ldots,n\}^d }P_{n,h}(x)\left(\frac{n+a}{a}\right)^d\int_{Q_{n,h}(a)} f(v)\,dv.
   \end{equation}

If $d=1$, formulae \eqref{new3.6} and \eqref{new3.61} turn into
\begin{equation}\label{new3.7}
C_n(f)(x)=\sum_{h=0}^n \binom{n}{h}x^h(1-x)^{n-h}\int_0^1 f\left(\frac{h+as}{n+a}\right)ds
\end{equation}
and, if $a>0$,
\begin{equation}\label{new3.71}
C_n(f)(x)=\sum_{h=0}^n \binom{n}{h}x^h(1-x)^{n-h}\left(\frac{n+a}{a}\right)\int_{\frac{h}{n+a}}^{\frac{h+a}{n+a}} f(t)\,dt.
   \end{equation}

\begin{theorem}\label{approxLp}
Assume that $a>0$. If $f\in\mathscr{L}^p(Q_d)$, $1\leq p<+\infty$, then $\displaystyle\lim_{n\to \infty} C_n(f)=f$ in $\mathscr{L}^p(Q_d)$.
\end{theorem}

\begin{proof}
Since $\mathscr{C}(Q_d)$ is dense in $\mathscr{L}^p(Q_d)$ with respect to the $\mathscr{L}^p$-norm $\|\cdot\|_p$ and since, on account of Theorem  \ref{ApproxC(K)}, $\lim_{n\to \infty} C_n(f)=f$ in $\mathscr{L}^p(Q_d)$ for every $f\in\mathscr{C}(Q_d)$, it is enough to show that the sequence $(C_n)_{n\geq 1}$ is equibounded from   $\mathscr{L}^p(Q_d)$ into $\mathscr{L}^p(Q_d)$.

Fix, indeed, $f\in\mathscr{L}^p(Q_d)$, $n\geq 1$, and $x\in Q_d$. By recalling that the function $|t|^p$ ($t\in\mathbf{R}$) is convex and that 
\[
\left| \int_{Q_d} g(u)\,du \right|^p\leq \int_{Q_d} |g(u)|^p\,du
\]
for every $g\in \mathscr{L}^p(Q_d)$, setting $M:=\displaystyle \sup_{n\geq 1} \left(\frac{n+a}{a(n+1)}\right)^d$, on account of \eqref{new3.3} we get
\[
\begin{split}
&|C_n(f)(x)|^p\leq \sum_{h\in\{0,\ldots,n\}^d }P_{n,h}(x)\int_{Q_d}\left| f\left(\frac{h+au}{n+a}\right)\right|^p du\\
&= \sum_{h\in\{0,\ldots,n\}^d }P_{n,h}(x) \left(\frac{n+a}{a}\right)^d\int_{Q_{n,h}(a)}| f(v)|^p dv.
\end{split}
\]
 
Therefore, by using \eqref{new3.4} we obtain
\[
\begin{split}
&\int_{Q_d}|C_n(f)(x)|^p dx\leq \sum_{h\in\{0,\ldots,n\}^d } \left(\frac{n+a}{a(n+1)}\right)^d\int_{Q_{n,h}(a)}| f(v)|^p dv\\
&\leq M \int_{Q_d}| f(v)|^p dv,
\end{split}
\]
i.e., $\|C_n(f)\|_p\leq M^{1/p} \|f\|_p$, and so the result follows.
\end{proof}


We present some estimates of the rate of convergence in Theorem \ref{approxLp}. 
\begin{proposition}
For every Borel measurable bounded function $f$ on $Q_d$, $1\leq p<+\infty$ and $n\geq 1$,
\begin{equation}\label{stima1}
\|C_n(f)-f\|_p\leq C \tau\left(f; \sqrt[2d]{\frac{3n+a^2}{12(n+a)^2}}\right)
\end{equation}
(see 	\eqref{tau}) where the positive constant $C$ does not depend on $f$. 
\end{proposition}
\begin{proof} 
Let $f$ be a Borel measurable bounded function on $Q_d$ and $p\geq 1$. 

For a fixed $x\in Q_d$, set $\Psi_x(y):=y-x$ for every $y\in Q_d$. 
By virtue of \cite[Remark p. 285]{Quak}, defining $M:=\sup\{C_n((pr_i\circ \Psi_x)^2)(x)\,|\, i=1,\ldots,d, x\in Q_d\}$, then there exists a constant $C$ such that
\[
\|C_n(f)-f\|_p\leq C\tau(f; \sqrt[2d]{M})_p
\]
provided $M \leq 1$.

Therefore, in order to obtain the desired result  it is enough  to estimate $M$. 
Since for every $n\geq 1$, $i=1,\ldots,d$ and $x\in Q_d$, we have $(pr_i\circ \Psi_x)^2=pr_i^2-2x_ipr_i+x_i\mathbf{1}$,  then
\[
C_n((pr_i\circ \Psi_x)^2)(x)=C_n(pr_i^2)(x)-2x_iC_n(pr_i)(x)+x_i^2C_n(\mathbf{1})(x).
\]
One has $C_n(\mathbf{1})(x)=1$ and, from Lemma \ref{lemma1},
\[\begin{split}
&C_n(pr_i)(x)=\frac{a}{2(n+a)}+\frac{n}{n+a}x_i.
\end{split}\]

Moreover, since  $B_n(pr_i^2)=\dfrac{1}{n}x_i+\dfrac{n-1}{n} x_i^2$ (see \eqref{affinisquares} and \eqref{projh1}), then
\[\begin{split}
&C_n(pr_i^2)(x)=
\frac{a^2}{3(n+a)^2}+\frac{n(a+1)}{(n+a)^2} x_i+\frac{n(n-1)}{(n+a)^2}  x_i^2.
\end{split}\] 
Then
\[\begin{split}
&C_n((pr_i\circ \Psi_x)^2)(x)=\frac{a^2-n}{(n+a)^2}x_i^2+\frac{n-a^2}{(n+a)^2}x_i+\frac{a^2}{3(n+a)^2}\\&=\frac{n-a^2}{(n+a)^2}x_i(1-x_i)+\frac{a^2}{3(n+a)^2}\leq \frac{n-a^2}{4(n+a)^2}+\frac{a^2}{3(n+a)^2}= \frac{3n+a^2}{12(n+a)^2}.
\end{split}\] 
Therefore $M \leq \dfrac{3n+a^2}{12(n+a)^2} \leq 1$ and the result follows. 
\end{proof}

Now we present some estimates of the approximation error $\|C_n(f)-f\|_p$ by applying the results contained in \cite{berens}. To this end for every $n\geq 1$ and $p\in[1,+\infty[$, we have to estimate the quantity $\lambda_{n,p}$ defined by
\begin{equation}\label{lambdap}
\lambda_{n,p}:=\max_{0\leq i\leq d+1}\|C_n(\varphi_i)-\varphi_i\|_p,
\end{equation}
where the functions $\varphi_i$ are defined by \eqref{phyfunctions}.

Note that $\|C_n(\mathbf{1})-\mathbf{1}\|_p=0$. For every $i=1,\ldots,d$, by virtue of Lemma \ref{lemma1} we get
\[
C_n(pr_i)(x)-pr_i(x)=\frac{a}{n+a}\left(\frac{1}{2}-x_i\right),
\]
therefore
\[
\begin{split}
&\|C_n(pr_i)-pr_i\|_p=\frac{a}{n+a}\left(\int_{Q_d}\left|\frac{1}{2}-x_i\right|^pdx\right)^{1/p}\\&=\frac{a}{n+a}\left(\int_{0}^1\left|\frac{1}{2}-x_i\right|^pdx_i\right)^{1/p}
=\frac{a}{2(n+a)(p+1)^{1/p}}\leq \frac{a}{4(n+a)}.
\end{split}
\]
Moreover, 
\[
\begin{split}
&\|C_n(\varphi_{d+1})-\varphi_{d+1}\|_p\\
&=\frac{1}{(n+a)^2}\left(\int_{Q_d} \left|\frac{a^2d}{3}+n(a+1)\sum_{i=1}^dx_i-(n+2na+a^2)\sum_{i=1}^dx_i^2\right|^pdx\right)^{1/p}\\
&\leq\frac{1}{(n+a)^2}\left(\int_{Q_d} \left(\frac{a^2d}{3}+n(a+1)d+(n+2na+a^2)d\right)^pdx\right)^{1/p}\\
&\leq\frac{d}{(n+a)^2} \left(\frac{a^2}{3}+n(a+1)+(n+2na+a^2)\right)\\
&\leq \frac{d}{n+a}\left(\frac{a^2}{3}+(a+1)+(a+1)^2\right)\leq \frac{3d(a+1)^2}{n+a}.
\end{split}
\]
Hence
	\begin{equation}\label{stimalambdap1}
\lambda_{n,p}\leq  \frac{3d(a+1)^2}{n+a}.
\end{equation}

Consider the Sobolev space $\mathscr{W}_{\infty}^2(K)$  of all functions  $f\in\mathscr{L}^{\infty}(K)$ such that, for every $|k|\leq 1$, $D^kf$ exists (in the Sobolev sense) and $D^kf\in\mathscr{L}^{\infty}(K)$, endowed with the norm $\|f\|_{2,\infty}:=\max_{|k|\leq 2}\|D^kf\|_{\mathscr{L}^{\infty}(K)}$. 

\begin{proposition}
If $f\in W_\infty^2(Q_d)$ then, for every $1\leq p<+\infty$ and $n\geq 1$,
\begin{equation}\label{stimafirst1}
\|C_n(f)-f(x)\|_p\leq C\|f\|_{2,\infty}\lambda_{n,p}\leq \tilde{C}\|f\|_{2,\infty}\frac{1}{n+a},
\end{equation}
where the constants $C$ and $\tilde{C}$ do not depend on $f$. 

Moreover, if $f\in \mathscr{L}^1(Q_d)$, then, for every $n\geq 1$,
\begin{equation}\label{stimasecond1}\begin{split}
&\|C_n(f)-f(x)\|_1\leq C(\lambda_{n,1}\|f\|_{1}+\omega_{d+2,1} (f,\lambda_{n,1}^{1/(d+2)}))\\
& \leq C \left( \frac{3d(a+1)^2}{n+a}\|f\|_{1}+\omega_{d+2,1} \left(f,\left( \frac{3d(a+1)^2}{n+a}\right)^{1/(d+2)}\right)\right),
\end{split}\end{equation}
with $C$   not depending on $f$. 

\end{proposition}

\begin{proof}
Keeping \eqref{lambdap} and \eqref{stimalambdap1} in mind, estimate \eqref{stimafirst1} follows from Theorem 1 in \cite{berens}; moreover formula \eqref{stimasecond1} is a consequence of \cite[Theorem 2]{berens} and \eqref{stimalambdap1}.
\end{proof}

 In order to obtain a  result similar to  Theorem \ref{approxLp}  for the $d$-dimensional simplex $K_d$, $d\geq 1$, we shall adapt the  proof for $Q_d$ by making the necessary modifications.
 
 As usual, for every $n\geq 1$ and $h=(h_1,\ldots, h_d)\in\{0,\ldots,n\}^d$ we set $|h|=h_1+\ldots+h_d$. 
 
 For every $n\geq 1$, $h=(h_1,\ldots, h_d)\in\{0,\ldots,n\}^d$, $|h|\leq n$ and $x=(x_1,\ldots,x_d)\in K_d$ we set
 \begin{equation}\label{new3.8}
 P_{n,h}^*(x):= \frac{n!}{h_1!\cdots h_d!(n-h_1- \ldots-h_d)!} x_1^{h_1}\cdots x_d^{h_d}\left(1- \sum\limits_{i=1}^d
x_i\right)^{n-\sum\limits_{i=1}^d h_i}.
 \end{equation}
 
 Then
  \begin{equation}\label{new3.11}
 P_{n,h}^*\geq 0\quad \textrm{and} \quad\sum_{h\in\{0,\ldots,n\}^d,\atop  |h|\leq n} P_{n,h}=1\quad \textrm{on}\,\,K_d,
 \end{equation}
and, on account of \cite[Section 976]{E}, 
 \begin{equation}\label{new3.12}
\int_{K_d} P_{n,h}^*(x)\,dx=\frac{n!}{(n+d)!} =\frac{1}{(n+1)(n+2)\cdots (n+d)} \leq \frac{1}{(n+1)^d}.
\end{equation}

Moreover,  for any $a\geq 0$, we set
 \begin{equation}\label{new3.9}
 \begin{split}
K_{n,h}(a):= \left\{ (x_1,\ldots,x_d)\in \mathbf{R}^d\;\big|\; \frac{h_i}{n+a}\leq x_i \textrm{ for each } i=1,	\ldots,d\ \textrm{ and }\right.\\ \left. \sum_{i=1}^d x_i\leq \frac{1}{n+a}\left(a+\sum_{i=1}^d h_i\right)\right\};
\end{split}
 \end{equation}
in particular,
 \begin{equation}\label{new3.10}
K_{n,h}(a)\subset K_d\quad\textrm{and}\quad K_d=\bigcup_{h\in\{0,\ldots,n\}^d,\atop  |h|\leq n} K_{n,h}(a).
 \end{equation}

Consider now the operators $C_n$, $n\geq 1$, defined by \eqref{new2.15} where each $\mu_n$ is the normalized Borel-Lebesgue measure $d!\lambda_d$ on $K_d$. Also in this case the operators $C_n$ are well-defined on $ \mathscr{L}^1(K_d)$ and the operator \eqref{pros} satisfies \eqref{ass1}. Moreover, for every $f\in \mathscr{L}^1(K_d)$ and $x\in K_d$,
\begin{equation}\label{new3.13}
C_n(f)(x)=\sum_{h\in\{0,\ldots,n\}^d,\atop  |h|\leq n} d!P_{n,h}^*(x)\int_{K_d} f\left(\frac{h+au}{n+a}\right) du
\end{equation}
and, if $a>0$,
\begin{equation}\label{new3.131}
C_n(f)(x)= \sum_{h\in\{0,\ldots,n\}^d,\atop  |h|\leq n}P_{n,h}^*(x) d!\left(\frac{n+a}{a}\right)^d\int_{K_{n,h}(a)}f(v)\,dv.
\end{equation}

Therefore, if $1\leq p<+\infty$ and $a>0$,
\[\begin{split}
&\int_{K_d}|C_n(f)(x)|^pdx\leq \sum_{h\in\{0,\ldots,n\}^d,\atop  |h|\leq n}\int_{K_d}P_{n,h}^*(x)\,dx\,\,  d!\left(\frac{n+a}{a}\right)^d\int_{K_{n,h}(a)}|f(v)|^p\,dv\\
&\leq  d!\left(\frac{n+a}{a(n+1)}\right)^d\int_{K_d}|f(v)|^p\,dv,
\end{split}
\]
and hence, setting $\overline{M}:=\displaystyle\sup_{n\geq 1} d!\left(\frac{n+a}{a(n+1)}\right)^d$, we get 
\[
\|C_n(f)\|_p\leq \overline{M}^{1/p}\|f\|_p.
\] 

By reasoning as in the proof of Theorem \ref{approxLp} we infer that 
\begin{theorem}\label{approcLpbis}
If $a>0$, then, for every $f\in\mathscr{L}^p(K_d)$, $1\leq p<+\infty$, $\displaystyle\lim_{n\to \infty}C_n(f)=f$ in $\mathscr{L}^p(K_d)$.
\end{theorem}

Also for  Theorem \ref{approcLpbis} we shall furnish some estimates of the rate of convergence 
by applying the results contained in \cite{berens}. In order to do this, first we evaluate $\lambda_{n,p}$ for each $n	\geq 1$ and $1\leq p<+\infty$ (see \eqref{lambdap}).

First of all, we recall that (see \cite[Section 976]{E}), for every $k \in \mathbf{N}$,
\begin{equation}\label{intprip}
\int_{K_d}x_i^k dx=\frac{\Gamma(k+1)}{\Gamma(k+d)}\frac{1}{k+d}=\frac{1}{(k+d)(k+d-1)\cdots(k+1)}.
\end{equation}

As in the case of hypercube,  $\|C_n(\mathbf{1})-\mathbf{1}\|_p=0$. For every $i=1,\ldots,d$, by virtue of Lemma \ref{lemma1} and \eqref{intprip} for $k=0, 1$, we get that, for every $x=(x_1, \ldots, x_d) \in K_d$,
\[
|C_n(pr_i)(x)-pr_i(x)|=\frac{a}{n+a}\left|\frac{1}{d+1}-x_i\right| \leq \frac{a}{n+a};
\]
therefore
\[
\|C_n(pr_i)-pr_i\|_p^\leq  \frac{a}{(n+a)(d!)^{1/p}}.\]

Moreover, from \eqref{affinisquares} and \eqref{pros} it follows that $B_n(pr_i^2)=\dfrac{n-1}{n}pr_i^2+\dfrac{1}{n}pr_i$. Thus, from Lemma \ref{lemma1} and \eqref{intprip} for $k=2$, for every $x=(x_1, \ldots, x_d) \in K_d$,
we get

\[\begin{split}
&C_n(pr_i^2)(x)=
\frac{a^2d!}{(n+a)^2}\frac{2}{(d+2)!}+\frac{2nad!}{(n+a)^2} \frac{1}{(d+1)!}x_i\\&+\frac{n^2}{(n+a)^2}\left[\frac{1}{n}x_i+\frac{n-1}{n} x_i^2\right]
\\&=\frac{2a^2}{(n+a)^2(d+2)(d+1)}+\frac{n(2a+d+1)}{(n+a)^2(d+1)} x_i+\frac{n(n-1)}{(n+a)^2} x_i^2
\end{split}\] 
 and hence (see \eqref{phyfunctions})
\[
\begin{split}
&|C_n(\varphi_{d+1})(x)-\varphi_{d+1}(x)|=
\frac{1}{(n+a)^{2}}  \\&\times \left|\frac{2a^2d}{(d+2)(d+1)}+\frac{n(2a+d+1)}{ (d+1)} \sum_{i=1}^dx_i-(n+2na+a^2)\sum_{i=1}^dx_i^2\right|\\
&\leq \frac{1}{(n+a)^{2}}  \left(\frac{2a^2d}{(d+2)(d+1)}+\frac{n(2a+d+1)}{ (d+1)} +n+2na+a^2\right).
\end{split}
\]

Thus
\[
\begin{split}
&\|C_n(\varphi_{d+1})-\varphi_{d+1}\|_p  \\
&\leq  \frac{1}{(n+a)^{2} (d!)^{1/p}}  \left(\frac{2a^2d}{(d+2)(d+1)}+\frac{n(2a+d+1)}{ (d+1)} +n+2na+a^2\right)\\
& \leq   \frac{2a^2+2d(a+1)+(d+1)(a+1)^2}{(n+a) (d!)^{1/p}(d+1)} \leq  \frac{3d+3}{ (d!)^{1/p}(d+1)} \frac{(a+1)^2}{n+a}.
\end{split}
\]

Accordingly,
	\begin{equation}\label{stimalambdap}
\lambda_{n,p}\leq   \frac{3(a+1)^2}{(d!)^{1/p}(n+a)}.
 \end{equation}

As a consequence, we get the following result.

\begin{proposition}
If $f\in W_\infty^2(K_d)$ then, for every  $n\geq 1$,
\begin{equation}\label{stimafirst}
\|C_n(f)-f(x)\|_p\leq C\|f\|_{2,\infty}\lambda_{n,p}\leq \tilde{C}\|f\|_{2,\infty}\frac{1}{n+a},
\end{equation}
where the constants $C$  and $\tilde{C}$ do  not depend on $f$.

Moreover, if $f\in \mathscr{L}^1(K_d)$, then  for every $1\leq p<+\infty$ and $n\geq 1$,
\begin{equation}\label{stimasecond}\begin{split}
&\|C_n(f)-f(x)\|_1\leq C(\lambda_{n,1}\|f\|_{1}+\omega_{d+2,1} (f,\lambda_{n,1}^{1/(d+2)}))\\
& \leq C \left(   \frac{3(a+1)^2}{d!(n+a)}\|f\|_{1}+\omega_{d+2,1} \left(f,\left(   \frac{3(a+1)^2}{d!(n+a)}\right)^{1/(d+2)}\right)\right)
\end{split}\end{equation}
with $C$  not depending on $f$. 

\end{proposition}

\section{Preservation properties}

\noindent In this section we shall investigate  some shape and regularity preserving properties of the $C_n$'s proving that, under suitable
assumptions, they preserve convexity and Lipschitz-continuity.

Relation \eqref{newexp} yields that some of the preservation
properties of the $B_n$'s are naturally shared by the $C_n$'s.

First of all we investigate  the behavior of the sequence
$(C_n)_{n \geq 1}$ on  Lipschitz-continuous functions and, to this
end, we recall some basic definitions.

Here we shall assume that $K$ is metrizable and we denote by  $\rho$  the metric on $K$ which induces its topology.
The $\rho$-modulus of continuity of a given $f \in \mathscr{C}(K)$ with respect to $\delta>0$ is then defined by
\begin{equation}\label{omegarho}
\omega_{\rho}(f,\delta):=\sup\{|f(x)-f(y)|\, \mid \,x,y\in K, \rho(x,y)\leq \delta\}.
\end{equation}

Furthermore, for any
$M\geq 0$ and $0<\alpha\leq 1$, we denote by
\begin{equation}\label{lipschitz}
\textrm{Lip} (M,\alpha):=\{f\in \mathscr{C}(K)\, \mid \, |f(x)-f(y)|\leq M\rho(x,y)^{\alpha}\textrm{ for every }x,y\in K\}
\end{equation}
the space of all H\"{o}lder continuous functions with exponent $\alpha$ and constant $M$. In particular, Lip$(M, 1)$ is  the space of all Lipschitz continuous functions with constant $M$. Assume that
\begin{equation}\label{hpthm1.6.6}
\omega_{\rho}(f,t\delta)\leq (1+t)\omega_{\rho}(f,\delta)
\end{equation}
for every $f\in \mathscr{C}(K)$, $\delta,t>0$.

  From now on we  suppose that    there exists $c \geq 1$ such that
\begin{equation}\label{condlip}
T(\textrm{Lip}( 1, 1) )\subset \textrm{Lip}(c,1),
\end{equation}
or, equivalently,
\begin{equation}\label{condlip2}
T(\textrm{Lip}( M, 1) )\subset \textrm{Lip}(cM,1),
\end{equation}
for every $M\geq 0$.

For instance, the Markov operators $T_1$, $S_d$ and $T_d$ of parts $1$, $2$ and $3$ of Examples \ref{esCn}, satisfy condition \eqref{condlip} with $c=1$, by considering on $[0,1]$ the usual metric and on $Q_d$ and $K_d$ the $l_1$-metric, i.e., the metric generated by the $l_1$-norm (see \cite[p. 124]{ACMLR2014}).

Under  condition \eqref{condlip} it has been shown that (see \cite[Theorem 3.3.1]{ACMLR2014}) 
\begin{equation}\label{lip1}
B_n(f) \in \textrm{Lip}(cM, 1),
\end{equation}
for every $f \in \textrm{Lip}(M,1)$ and $n \geq 1$. Moreover (see \cite[Corollary 3.3.2]{ACMLR2014}), for every $f \in \mathscr{C}(K)$, $\delta >0$ and $n \geq 1$,
\begin{equation}\label{lip2}
\omega_{\rho}(B_n(f), \delta) \leq (1+c) \omega_{\rho}(f, \delta)
\end{equation}
and
\begin{equation}\label{lip3}
B_n(\textrm{Lip}(M, \alpha)) \subset \textrm{Lip}(c^\alpha M, \alpha)
\end{equation}
for every $M>0$ and $\alpha\in]0,1]$.

Then, taking \eqref{lip1} into account, after noting that $I_n(f)
\in\textrm{ Lip}(M, \alpha)$  whenever $f \in \textrm{Lip}(M, \alpha)$ (see
\eqref{fn}),
 we easily deduce the following
proposition.

\begin{proposition}\label{proplipschitz}
Assume that condition \eqref{condlip} is satisfied. Then, for every $f
\in\mathrm{Lip}(M, 1)$ and $n \geq 1$, $ C_n(f) \in \mathrm{Lip}\left(cM, 1
\right)$.
\end{proposition}

Moreover,  \eqref{lip2} and \eqref{lip3} yield the following
result, since $\omega_\rho(I_n(f),\delta)\leq \omega_\rho(f,\delta)$. 
\begin{proposition}
Assume  that condition  \eqref{condlip}  is  satisfied. Then, for
every $f \in \mathscr{C}(K)$, $\delta >0$ and $n \geq 1$,
\begin{equation*}
\omega_{\rho}(C_n(f), \delta) \leq (1+c)  \ \omega_{\rho}\left(f,  \delta\right)
\end{equation*}
and
\begin{equation*}
 C_n(\mathrm{Lip}(M, \alpha)) \subset \mathrm{Lip}(c^\alpha
M, \alpha)
\end{equation*}
for every $M>0$ and $0<\alpha\leq 1$.

In particular, if $T(\mathrm{Lip}(1,1)) \subset \mathrm{Lip}(1,1)$, then
\begin{equation*}
\omega_{\rho}(C_n(f), \delta) \leq 2  \ \omega_{\rho}\left(f,  \delta\right)
\end{equation*}
and
\begin{equation*}
 C_n(\mathrm{Lip}(M, \alpha)) \subset \mathrm{Lip}(M,
\alpha)
\end{equation*}
for every $M>0$ and $0<\alpha\leq 1$.
\end{proposition}

 Below we state some further properties of the operators $C_n$'s for special functions $f\in\mathscr{C}(K)$. To this end we need some additional concepts.

 We
 first recall that, if $T$ is an arbitrary Markov operator on $\mathscr{C}(K)$,   a function $f \in \mathscr{C}(K)$ is said to be $T$-convex
 if
 \begin{equation*}
 f_{z, \alpha} \leq T(f_{z, \alpha}) \ \quad \ \textrm{for every} \ z
 \in K \ \textrm{and} \ \alpha \in [0,1],
 \end{equation*}
where $f_{z, \alpha}$ is defined by 
$f_{z,\alpha}(x):=f(\alpha x+(1-\alpha)z)$ ($x\in K$). 

If $K=[0,1]$ and $T_1$ denotes the operator \eqref{defT1}, then a function $f\in\mathscr{C}([0,1])$ is $T_1$-convex if and only if it  is convex. 

For $K=Q_d$ and $T=S_d$ (see \eqref{projh1}), then a function $f\in\mathscr{C}(Q_d)$ is $S_d$-convex if and only if $f$ is convex with respect to each variable. 

Finally, for $K=K_d$ and $T=T_d$ (see \eqref{pros}), a function $f\in\mathscr{C}(K_d)$ is $T_d$-convex  if and only if it is axially convex, i.e., it is convex on each segment parallel to a segment joining  two extreme points of $K_d$ (see \cite[Section 3.5]{ACMLR2014} for more details). 

In general, each convex function $f \in \mathscr{C}(K)$ is $T$-convex. Moreover $T$-axially convex functions are $T$-convex as well (see \cite[Definition 3.5.1 and remarks on  p. 148]{ACMLR2014}).

In \cite[Theorem 3.5.2]{ACMLR2014} it has been showed that, if
$(B_n)_{n \geq 1}$ is the sequence of Bernstein-Schnabl operators associated with $T$ and $T$ satisfies hypothesis \eqref{ass1} (or \eqref{ass1a}), then
\begin{equation*}
f \leq B_n(f) \leq T(f)  \qquad(n \geq 1).
\end{equation*}
 whenever $f \in \mathscr{C}(K)$ is $T$-convex. As a consequence 
 we have the following result.

\begin{proposition}\label{convex}
Under hypothesis \eqref{ass1} (or \eqref{ass1a}), if $f \in \mathscr{C}(K)$  is $T$-convex then, for any $n \geq 1$,
\begin{equation*}\label{limitation}
  C_n(f) \leq C_n(T(f)).
\end{equation*}

In particular if $f$ is $T$-convex and each $I_n(f)$ is $T$-convex, then for every $n\geq 1$,
\begin{equation*}
I_n(f) \leq C_n(f) \leq T(I_n(f)).
\end{equation*}
\end{proposition}

Apart from the case of the interval $[0,1]$ and the classical Bernstein operators, in general    Bernstein-Schnabl operators
do not preserve convexity. A simple counterexample is given by the
function $f:=|pr_1-pr_2|$ defined on the two-dimensional simplex
$K_2$ (see \cite[p. 468]{Sa-1991}). Hence, in general, the $C_n$'s do
not preserve convexity, too.
But   it is possible to determine sufficient conditions in order that the
$B_n$'s (and, hence, the $C_n$'s) preserve convexity.

For a given $f\in\mathscr{C}(K)$, we set
\begin{equation}\label{rasa2.1.17}
\tilde{f}(s,t):=f\left(\frac{s+t}{2}\right)\quad (s,t\in K)
\end{equation}
and
\begin{equation}\label{rasa2.1.18}
\begin{split}
&\Delta(\tilde{f};x,y):=\iint_{K^2} \tilde{f}(s,t)\,d\tilde{\mu}_x^T(s)d\tilde{\mu}_x^T(t)+\\&+\iint_{K^2} \tilde{f}(s,t)\,d\tilde{\mu}_y^T(s)d\tilde{\mu}_y^T(t)-2\iint_{K^2} \tilde{f}(s,t)\,d\tilde{\mu}_x^T(s)d\tilde{\mu}_y^T(t)\\
&=B_2(f)(x)+B_2(f)(y)-2\iint_{K^2} f\left(\frac{s+t}{2}\right)\,d\tilde{\mu}_x^T(s)d\tilde{\mu}_y^T(t)
\end{split}
\end{equation}
for every $x,y\in K$.

 \begin{theorem}
 Suppose that $T$ satisfies the following assumptions:
 \begin{itemize}
 \item[($c_1$)] $T$ maps continuous convex functions into (continuous) convex functions;
 \item[($c_2$)] $\Delta(\tilde{f}; x,y)\geq 0$ for every convex function $f\in\mathscr{C}(K)$ and for every $x,y\in K$.
 \end{itemize}
 Then each $C_n$ maps continuous convex functions into (continuous) convex functions.
 \end{theorem}
\begin{proof}
According to \cite[Theorem 3.4.3]{ACMLR2014}, under assumptions ($c_1$) and ($c_2$), each Bernstein-Schnabl operator $B_n$ maps continuous convex functions into (continuous) convex functions. Therefore the result follows from \eqref{newexp} taking into account that each $I_n(f)$ is convex provided that $f\in\mathscr{C}(K)$ is convex.
\end{proof}

\begin{remark}
In \cite[Remark 3.4.4 and Examples 3.4.5-3.4.11]{ACMLR2014} there are several examples of settings where conditions ($c_1$) and ($c_2$) are satisfied. This is the case, in particular, when $K=[0,1]$ and $T=T_1$ (see \eqref{defT1}). Therefore, all the operators defined by \eqref{new2.8}, \eqref{new2.9} and \eqref{new2.10} preserve the convexity.
\end{remark}

Finally we point out that, if $K=K_d$, $d\geq 1$, then the Bernstein operators on $\mathscr{C}(K_d)$, i.e., the Bernstein-Schnabl operators associated with the Markov operator \eqref{pros}, preserve the axial convexity  (\cite[Theorem 6.3.2]{ac}, \cite[Theorem 3.5.9]{ACMLR2014}). On the other hand, if $f\in\mathscr{C}(K_d) $ 
 is axially convex, then $I_n(f)$ is axially convex too for every $n\geq 1$. Therefore, on account of \eqref{newexp}, we conclude that
\begin{corollary}
Considering the canonical simplex $K_d$ of $\mathbf{R}^d$, $d\geq 1$, then the operators $C_n$ defined by \eqref{new2.15} map continuous axially convex functions on $K_d$ into (continuous) axially convex functions.
\end{corollary}

 \bigskip
\noindent Francesco Altomare and Mirella Cappelletti Montano \\
Dipartimento di Matematica\\
Universit\`{a} degli Studi di Bari "A. Moro"\\
Campus Universitario, Via E. Orabona n. 4\\
70125-Bari, Italy\\
e-mail: francesco.altomare@uniba.it, mirella.cappellettimontano@uniba.it\\

\noindent Vita Leonessa\\
Dipartimento di Matematica, Informatica ed Economia\\
Universit\`{a} degli Studi della Basilicata\\
Viale Dell' Ateneo Lucano n. 10, Campus di Macchia Romana\\ 85100-Potenza,
Italy\\
e-mail: vita.leonessa@unibas.it \\

\noindent Ioan Ra\c{s}a\\
Department of Mathematics \\Technical University of Cluj-Napoca
\\Str. Memorandumului 28 \\RO-400114 Cluj-Napoca, Romania\\
e-mail: Ioan.Rasa@math.utcluj.ro\\

\end{document}